\documentclass[11pt,reqno]{amsart}

\usepackage{amsmath,amsfonts,amsthm,amscd,amssymb,graphicx,mathrsfs}
\usepackage[utf8]{inputenc}
\usepackage{amsbsy}
\usepackage{graphicx}
\usepackage{subcaption}
\usepackage{hyperref}
\usepackage[text={6.5in,9in},centering,includefoot,foot=0.6in]{geometry}

\usepackage{todonotes}

\definecolor{skyblue}{rgb}{0.85,0.85,1}

\newtheorem{theorem}{Theorem}

\newtheorem*{hyp*}{Hypothesis}

\newtheorem{conj}{Conjecture}

\renewcommand{\labelenumi}{(\roman{enumi})}

\newcommand{\bbN}{\mathbb{N}}             
\newcommand{\bbR}{\mathbb{R}}             

\newcommand{\p}{\partial}				

\DeclareMathOperator{\Mor}{Mor}

\newcommand{\cE}{{\mathcal{E}}}

\newcommand{\cI}{{\mathcal{I}}}

\begin{document}

\title{Isoperimetric relations between Dirichlet and Neumann eigenvalues}
\author[G. Cox]{Graham Cox}\email{gcox@mun.ca}
\author[S. MacLachlan]{Scott MacLachlan}\email{smaclachlan@mun.ca}
\author[L. Steeves]{Luke Steeves}\email{luke.steeves@mun.ca}
\address{Department of Mathematics and Statistics, Memorial University of Newfoundland, St. John's, NL A1C 5S7, Canada}

\maketitle
\begin{abstract}
Inequalities between the Dirichlet and Neumann eigenvalues of the Laplacian have received much attention in the literature, but open problems abound. Here, we study the number of Neumann eigenvalues no greater than the first Dirichlet eigenvalue. Based on a combination of analytical and numerical results, we conjecture that this number is controlled by the isoperimetric ratio of the domain. This has applications to the nodal deficiency of eigenfunctions and is closely related to a long-standing conjecture of Yau on the Hausdorff measure of nodal sets.
\end{abstract}


\section{Introduction}
Let $\Omega \subset \bbR^n$ be a bounded domain with sufficiently smooth boundary. Denote the Dirichlet and Neumann eigenvalues of $-\Delta$ by
\[
	0 < \lambda_1 < \lambda_2 \leq \lambda_3 \leq \cdots
\]
and
\[
	0 = \mu_1 < \mu_2 \leq \mu_3 \leq \cdots
\]
respectively. In the one-dimensional case, the interlacing property
\[
	\mu_k < \lambda_k \leq \mu_{k+1}
\]
holds for every $k$. The extent to which these inequalities generalize to higher dimensions has received much attention over the years. However, most results impose rather strong assumptions (e.g. convexity) on the geometry on $\p\Omega$, and the optimal results for general domains remain unknown.

\subsection{Survey of known results}
The variational formulation of the eigenvalue problem immediately gives $\mu_k \leq \lambda_k$ for all $k$. It was shown by P\'olya that $\mu_2 < \lambda_1$ always holds \cite{P52}. In the case of a convex planar domain with $C^2$ boundary, Payne \cite{P55} showed that $\mu_{k+2} < \lambda_k$ for any $k$.

By imposing assumptions on various combinations of the principal curvatures of the boundary, Levine and Weinberger \cite{LW86} showed that $\mu_{k+r} < \lambda_k$ for a $C^{2,\alpha}$ domain, where $1 \leq r \leq n$ depends on the geometric assumptions. In particular, $\mu_{k+n} < \lambda_k$ when $\Omega$ is convex. Through a limiting argument, they obtained $\mu_{k+n} \leq \lambda_k$ for a general (not necessarily smooth) convex domain. They also obtained (as a special case) a result of Aviles \cite{A86}, that $\mu_{k+1} < \lambda_k$ when the mean curvature is nonnegative.

More generally, it was shown by Friedlander \cite{F91} that $\mu_{k+1} \leq \lambda_k$ on any domain with $C^1$ boundary. This result was extended to Lipschitz domains by Arendt and Mazzeo \cite{AM12} and even further by Filinov \cite{F04} to domains of finite Lebesgue measure for which the embedding $H^1(\Omega) \subset L^2(\Omega)$ is compact. In fact, the latter two references obtain the strict inequality $\mu_{k+1} < \lambda_k$.

\subsection{Open problems and conjectures}
The question of whether or not $\mu_{k+n} \leq \lambda_k$ holds for non-convex domains remains open.\footnote{A purported counterexample in \cite{LW86} is easily seen to be wrong, as has been pointed out in \cite{BLP09}.} More precisely, given a Dirichlet eigenvalue $\lambda_k$, one can ask how many Neumann eigenvalues, $\mu$, exist with $\mu \leq \lambda_k$, and how this number depends on the geometry of $\Omega$. We investigate this when $k=1$, where little is known beyond the general results described above.

The case $k=1$ is particularly important when studying nodal domains of eigenfunctions. This is because any eigenvalue $\lambda$ is necessarily the first Dirichlet eigenvalue on any nodal domain of its associated eigenfunction. Several applications of this idea are described in Section \ref{sec:apply}.

Given a domain $\Omega$, we define the number
\begin{align}\label{eq:Ndef}
	N(\Omega) = \#\{k \in \bbN : \mu_k \leq \lambda_1\}.
\end{align}
It follows from the results surveyed above that $N(\Omega) \geq 2$ for any domain, and $N(\Omega) \geq n+1$ when $\Omega$ is convex. 

To obtain further intuition, we use the finite-element method (FEM) to approximate $N(\Omega)$ for a wide variety of planar domains. In every case, our finite-element approximations satisfy $N(\Omega) \geq 3$, whether or not $\Omega$ is convex. In fact, the farther $\Omega$ is from a ball, the larger $N(\Omega)$ is seen to be, bringing to mind the remark of Osserman \cite{O78} that ``One has the somewhat ironic situation that the more irregular the boundary, the stronger will be the isoperimetric inequality, but the harder it is to prove."

We quantify this irregularity using the isoperimetric ratio
\begin{align}
	\cI(\Omega) = \frac{|\p\Omega|^2}{|\Omega|},
\end{align}
where $|\Omega|$ denotes the area of the domain and $|\p\Omega|$ the perimeter. Computing $N(\Omega)$ and $\cI(\Omega)$ numerically for a wide variety of examples, we arrive at the following conjecture.

\begin{figure}
	\includegraphics[scale=0.5]{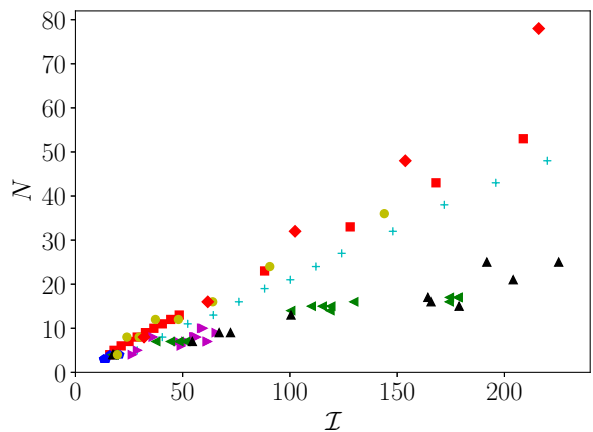}
\caption{$N$ vs. $\cI$ for planar domains of varying geometry and topology}
\label{fig:combined}
\end{figure}

\begin{conj}\label{con:2D}
There exist constants $c_1,c_2 > 0$ so that
\[
	c_1 \cI(\Omega) \leq N(\Omega) \leq c_2 \cI(\Omega)
\]
for any bounded Lipschitz domain $\Omega \subset \bbR^2$.
\end{conj}

The computations leading to this conjecture are described in Sections \ref{sec:analytical} and \ref{sec:numerical} and summarized in Figure \ref{fig:combined}.

While our numerical investigations focus on $n=2$, we conjecture that the same result also holds in higher dimensions. For $\Omega \subset \bbR^n$ we define the isoperimetric ratio
\begin{align}
	\cI(\Omega) = \frac{|\p\Omega|^n}{|\Omega|^{n-1}},
\end{align}
where $|\Omega|$ is the Lebesgue measure of $\Omega$, and $|\p\Omega|$ is the $(n-1)$-dimensional Hausdorff measure of the boundary.

\begin{conj}\label{con:general}
There exist constants $c_1,c_2 > 0$, depending only on $n$, so that
\[
	c_1 \cI(\Omega) \leq N(\Omega) \leq c_2 \cI(\Omega)
\]
for any Lipschitz domain $\Omega \subset \bbR^n$.
\end{conj}


In Theorem \ref{thm:rectangle}, we verify this conjecture for $n$-dimensional rectangles. For the unit ball, $B^n \subset \bbR^n$, we prove in Theorem \ref{thm:ball} that $N(B^n)$ grows faster than any polynomial function of $n$. The isoperimetric ratio satisfies
\[
	\cI(B^n) = \frac{(n\omega_n)^{n}}{\omega_n^{n-1}} = n^n \omega_n,
\]
where $\omega_n = \pi^{n/2} / \Gamma(n/2+1)$ is the volume of the unit ball in $\bbR^n$, and Stirling's approximation for the Gamma function implies $\omega_n \approx C n^{-(n+1)/2}$ for large $n$.  Therefore, $\cI(B^n) \approx C n^{(n-1)/2}$ also has superpolynomial growth in $n$.

To the best of our knowledge, none of the known isoperimetric inequalities for Dirichlet and Neuman eigenvalues are of use in resolving Conjecture \ref{con:general}. Instead, we suggest that more detailed information be sought in the spectrum of the Dirichlet-to-Neumann map, building on the work of Friedlander \cite{F91}. As described in Sections \ref{sec:nodal} and \ref{sec:size}, such information would also improve our understanding of Courant's nodal domain theorem and Yau's conjecture of the size of the nodal set.

\subsection*{Acknowledgments}
The authors would like to thank Ram Band for helpful discussions and comments on this work. G.C. acknowledges the support of NSERC grant RGPIN-2017-04259.  The work of S.M. was partially supported by NSERC discovery grants RGPIN-2014-06032 and RGPIN-2019-05692. L.S. was supported by an NSERC Undergraduate Summer Research Award and the aforementioned NSERC grants.

\section{Motivation and implications}\label{sec:apply}
In this section, we explain our motivation for considering the quantity $N(\Omega)$ in the first place, and describe the consequences Conjecture \ref{con:general} would have if true.

\subsection{Relation to the nodal deficiency}\label{sec:nodal}
Our investigation into the quantity $N(\Omega)$ was inspired by a recent result in \cite{CJM17} (see also \cite{BCM19}) on the nodal deficiency of Laplacian eigenfunctions. Combining this with Friedlander's lemma for the Dirichlet-to-Neumann map, we obtain an upper bound on the spectral position of an eigenvalue.

To state the result in its simplest form, let $\phi$ be a Dirichlet eigenfunction of the Laplacian on $\Omega$, corresponding to a simple eigenvalue $\lambda$, and define the regions $\Omega_+ = \{\phi(x) > 0\}$ and $\Omega_- = \{\phi(x) < 0\}$. Since $\lambda$ is simple, it equals $\lambda_k$ for a unique $k \in \bbN$, which we call the spectral position of $\lambda$.

\begin{theorem}\label{thm:position}
The spectral position of $\lambda$ satisfies
\begin{align}\label{kbound}
	k \leq N(\Omega_+) + N(\Omega_-),
\end{align}
where the indices on the right-hand side count eigenvalues with either Dirichlet or Neumann boundary conditions on the interior boundaries $\p\Omega_\pm \cap \Omega$ and Dirichlet boundary conditions on $\p\Omega$.
\end{theorem}

This formula remains valid if $\Omega$ is replaced by a compact Riemannian manifold, and a modified version holds for degenerate eigenvalues; see \cite{BCM19} for details.

\begin{proof}
Let $\nu_\pm$ denote the number of connected components of $\Omega_\pm$, so the total number of nodal domains is $\nu = \nu_+ + \nu_-$. On $\Omega_\pm$, the first Dirichlet eigenvalue is $\lambda$, with multiplicity $\nu_\pm$. Thus, for sufficiently small $\epsilon > 0$, $0$ is not an eigenvalue of the operator $\Delta + (\lambda+\epsilon)$ on $\Omega_\pm$, so the Dirichlet-to-Neumann maps $\Lambda_\pm$ are well defined (see \cite{CJM17} for details).

The nodal deficiency of $\phi$ is defined to be $\delta(\phi) = k - \nu$. From \cite[Corollary 2.5]{CJM17}, we have $\delta(\phi) = \Mor(\Lambda_+ + \Lambda_-) \leq \Mor(\Lambda_+) + \Mor(\Lambda_-)$, where $\Mor$ denotes the Morse index, i.e. the number of negative eigenvalues. Moreover, \cite[Lemma 1.2]{F91} implies  $\Mor(\Lambda_\pm) = N(\Omega_\pm) - \nu_\pm$ when $\epsilon$ is sufficiently small, so we obtain $
k - \nu  \leq \big(N(\Omega_+) - \nu_+\big) + \big(N(\Omega_-) - \nu_-\big)$ and the result follows.
\end{proof}

If Conjecture \ref{con:general} is true, then Theorem \ref{thm:position} implies that highly deficient eigenfunctions have nodal domains with large isoperimetric ratios. For instance, if $\phi_k$ is an eigenfunction with only two nodal domains, then one of them, say $\Omega_+$, must have $N(\Omega_+) \geq k/2$ and hence $\cI(\Omega_+) \geq c_1k/2$, where the constant $c_1$ only depends on the dimension. This is particular useful when one has a subsequence of eigenfunctions $(\phi_{j_{_k}})$ each having just two nodal domains; see \cite{CH53,L77,S25} for classical examples of such eigenfunctions, and \cite{JZ18} for a recent construction.

In general, denoting the nodal domains by $\Omega_1, \ldots, \Omega_\nu$, we have
\[
	k \leq N(\Omega_1) + \cdots +  N(\Omega_\nu) = \nu \overline{N(\Omega_\ast)},
\]
where $\overline{N(\Omega_\ast)}$ denotes the mean of the numbers $N(\Omega_1), \ldots, N(\Omega_\nu)$. Generalizing the two-dimensional result of Pleijel \cite{P56}, B\'erard and Meyer showed in \cite{BM82} that
\[
	\limsup_{k \to \infty} \frac{\nu(\phi_k)}{k} \leq \gamma(n)
\]
where $\gamma(n) < 1$ is a universal constant that only depends on the dimension. From this we obtain
\begin{align}\label{eq:Pleijel}
	\liminf_{k \to \infty} \overline{N(\Omega_\ast)} \geq \frac{1}{\gamma(n)}.
\end{align}
When $n=2$ we have $\gamma(2) = (2/j_{0,1})^2 < 0.7$, so \eqref{eq:Pleijel} yields $\liminf \overline{N(\Omega_\ast)} \geq 1.4$. This is not anything new, since Friedlander has already shown that $N(\Omega_i) \geq 2$ for every $i$, so the same is true of the mean. Thus, for \eqref{eq:Pleijel} to be of any use, one must have $\gamma(n) < 1/2$, as this would imply $\liminf \overline{N(\Omega_\ast)} > 2$. It was shown in \cite{HP16} that the constant $\gamma(n)$ is strictly decreasing in $n$, with $\gamma(n) < 1/2$ for $n \geq 3$. In fact, one has $\gamma(n) \leq C (2/e)^n$ for some uniform constant $C>0$.

\subsection{The size of the nodal set}\label{sec:size}
Recall that $\Sigma = \{\phi(x) = 0\}$ denotes the nodal set of a fixed eigenfunction. A well-known conjecture of Yau suggests that
\begin{align}\label{Yau}
	c_1 \sqrt\lambda \leq |\Sigma| \leq c_2 \sqrt\lambda,
\end{align}
where $\lambda$ is the eigenvalue corresponding to $\phi$. This result is known to be true when $n=2$ \cite{B78}, and on real analytic manifolds \cite{DF88}. Recently, Logunov showed that the lower bound on $|\Sigma|$ holds on a smooth manifold of any dimension \cite{L18lower}, and proved the upper bound $|\Sigma| \leq c_2 \lambda^\alpha$ for some constant $\alpha>1/2$ that depends only on $n$ \cite{L18upper}.

We observe here that the upper bound in \eqref{Yau} follows from Conjecture \ref{con:general} and an additional assumption on the nodal domains and nodal sets. Recall that $\Omega_1, \ldots, \Omega_\nu$ denote the nodal domains of an eigenfunction. If we assume that there exist constants $a$ and $b$ such that $\nu |\Omega_i| \geq a$ and $\nu |\p\Omega_i| \leq b |\Sigma|$, uniformly in $\lambda$, then it follows from Theorem \ref{thm:position} and Conjecture \ref{con:general} that
\begin{align*}
	k &\leq N(\Omega_1) + \cdots + N(\Omega_\nu) \\
	&\leq c_2 \left( \frac{|\p\Omega_1|^n}{|\Omega_1|^{n-1}} + \cdots + \frac{|\p\Omega_\nu|^n}{|\Omega_\nu|^{n-1}} \right) \\
	&\leq c_2 b^n a^{1-n} |\Sigma|^n.
\end{align*}
Furthermore, Weyl's law implies $\lambda_k^{n/2} \leq C k$ for some uniform constant $C$, so we obtain $\lambda_k^{n/2} \leq C |\Sigma|^n$, as desired. However, the assumed uniform bounds on the nodal domains and nodal sets are not known. 



We conclude by stating an equivalent form of Yau's conjecture in terms of the nodal deficiency. 
It follows from the argument of Pleijel \cite{P56} that $c_1 \lambda^{n/2} \leq \delta(\phi) \leq c_2 \lambda^{n/2}$, from which we see that \eqref{Yau} is, in fact, equivalent to the existence of uniform constants $c_1,c_2>0$ such that
\begin{align}
	c_1 |\Sigma|^n \leq \delta(\phi) \leq c_2 |\Sigma|^n
\end{align}
for all eigenfunctions $\phi$.
%
%
From this equivalence, Logunov's upper and lower bounds on the size of the nodal set immediately imply there exist positive constants $c_1,c_2$ and $\beta  = \beta(n)$ such that $c_1 |\Sigma|^{n/\beta} \leq \delta(\phi) \leq c_2 |\Sigma|^n$, where $\beta(2)= 1$ and $\beta(n) > 1$ for $n > 2$.

We find this formulation of Yau's conjecture particularly appealing in light of the formula $\delta(\phi) = \Mor(\Lambda_+ + \Lambda_-)$ for the nodal deficiency, which was proved in \cite{CJM17} and used in the proof of Theorem \ref{thm:position} above. In particular, it suggests an alternate route to studying the conjecture through the spectra of the Dirichlet-to-Neumann maps, $\Lambda_\pm$.

\section{Analytical results}\label{sec:analytical}

In this section, we consider a few separable examples, where the calculations are relatively explicit.

\subsection{Rectangles}\label{sec:rectangle}

We first consider an $n$-dimensional rectangle with side lengths $\ell_1, \ldots, \ell_n$. In this case, we are able to verify Conjecture \ref{con:general}.

\begin{theorem}\label{thm:rectangle}
There exist positive constants, $c_1(n)$ and $c_2(n)$, such that
\[
	c_1 \cI(R) \leq N(R) \leq c_2 \cI(R)
\]
for any rectangle $R \subset \bbR^n$.
\end{theorem}

\begin{proof}

The Dirichlet eigenvalues are $(m_1\pi/\ell_1)^2 + \cdots + (m_n\pi/\ell_n)^2$
for $m_i \in \bbN$, while the Neumann eigenvalues are given by the same formula but with $m_i \in \bbN_0 = \bbN \cup \{0\}$. Therefore
\begin{align}
	N(R) = \#\left\{ (m_1, \ldots, m_n) \in \bbN_0^n : \frac{m_1^2}{\ell_1^2} + \cdots + \frac{m_n^2}{\ell_n^2} \leq \frac{1}{\ell_1^2} + \cdots + \frac{1}{\ell_n^2}\right\}.
\end{align}

This equals the number of nonnegative lattice points in the $n$-ellipsoid with axes $a_i = \ell_i\rho$, where $\rho^2 = \ell_1^{-2} + \cdots + \ell_n^{-2}$.
Denote this ellipsoid by $\cE(a)$, and let $L = \cE(a) \cap \bbN_0^n$, so $N(R) = \# L$. For each $m = (m_1,\ldots,m_n) \in L$ let $C_m$ denote the unit cube with smallest vertex at $m$, 
and define the set
\[
	E = \bigcup_{m \in L} C_m,
\]
so that $N(R) = \# L = |E|$. Finally, let $P \subset \bbR^n$ denote the first quadrant, where all coordinates are nonnegative. We claim that
\begin{align}\label{Econtain}
	\cE(a) \cap P \subset E \subset \cE(2a) \cap P,
\end{align}
where $\cE(2a)$ is the ellipsoid with axes $2 a_i = 2 \ell_i \rho$.

Let $x = (x_1,\ldots,x_n) \in \cE(a) \cap P$ and define $m = \lfloor x \rfloor \in \bbN^n_0$ componentwise. Then
\[
	\frac{m_1^2}{\ell_1^2} + \cdots + \frac{m_n^2}{\ell_n^2} \leq \frac{x_1^2}{\ell_1^2} + \cdots + \frac{x_n^2}{\ell_n^2} \leq \rho^2,
\]
hence $m \in L$ and $x \in C_m \in E$. On the other hand, for $m \in L$ we have
\[
	\frac{(m_1 + 1)^2}{\ell_1^2} + \cdots + \frac{(m_n + 1)^2}{\ell_n^2} \leq 4 \rho^2,
\]
hence $C_m \subset \cE(2a) \cap P$. It follows that $E \subset \cE(2a) \cap P$, so we have verified \eqref{Econtain}, and conclude that the volume $|E|= N(R)$ satisfies
\begin{align}\label{rec:Nbound}
	\frac{\omega_n}{2^n} a_1 \cdots a_n \leq |E| \leq  \omega_n a_1 \cdots a_n,
\end{align}
where $\omega_n$ is the volume of the unit ball in $\bbR^n$.

We next estimate the isoperimetric ratio $\cI(R)$. The rectangle has volume $|R| = \ell_1 \cdots \ell_n$, and the boundary has area
\[
	|\p R| = 2 \sum_{i=1}^n \ell_1 \cdots \widehat \ell_i \cdots \ell_n,
\]
where $\ell_1 \cdots \widehat \ell_i \cdots \ell_n$ denotes the product with the $i$th term omitted. Note that
\[
	\frac{1}{\ell_1^2} + \cdots + \frac{1}{\ell_n^2} = \frac{\sum_{i=1}^n (\ell_1 \cdots \widehat \ell_i \cdots \ell_n)^2}{(\ell_1 \cdots \ell_n)^2}
\]
and hence
\[
	a_1 \cdots a_n = \frac{\left(\sum_{i=1}^n (\ell_1 \cdots \widehat \ell_i \cdots \ell_n)^2\right)^{n/2}}{(\ell_1 \cdots \ell_n)^{n-1}}.
\]

Using the inequalities
\[
	\frac1n \left( \sum_{i=1}^n x_i \right)^2 \leq \sum_{i=1}^n x_i^2 \leq \left( \sum_{i=1}^n x_i \right)^2
\]
for nonnegative numbers $(x_i)$, we obtain
\begin{align}\label{rec:Ibound}
	\frac{1}{2^n n^{n/2}} \frac{|\p R|^n}{|R|^{n-1}} \leq a_1 \cdots a_n \leq \frac{1}{2^n} \frac{|\p R|^n}{|R|^{n-1}}.
\end{align}

Combining \eqref{rec:Nbound} and \eqref{rec:Ibound} completes the proof.
\end{proof}

For a rectangle in $\bbR^2$ we can obtain a more precise result. Assume without loss of generality that the side lengths are $1$ and $\ell \geq 1$. The inequality
\[
	m^2 + \frac{n^2}{\ell^2} \leq 1 + \frac{1}{\ell^2}
\]	
is only satisfied by $(0,0)$, $(1,0)$, $(1,1)$ and $(0,n)$ with $1 \leq n^2 \leq 1 + \ell^2$, so we have
\begin{align}\label{rec:2D}
\begin{split}
	N(R) &= 3 + \lfloor \sqrt{\ell^2+1}\rfloor \\
	\cI(R) &= 4(1+\ell)^2/\ell.
\end{split}
\end{align}
In particular, we see that $N(R) / \cI(R) \to 1/4$ as $\ell \to \infty$.

%
%
%
%
%
%
%

\subsection{The unit ball}
Levine and Weinberger observed in \cite{LW86} that $\mu_3 < \lambda_1 < \mu_4$ on the disc in $\bbR^2$, and $\mu_4 < \lambda_1 < \mu_5$ for the ball in $\bbR^3$. Here, we investigate extensions of these inequalities to the unit ball, $B^n \subset \bbR^n$, for $n>3$. Since $B^n$ is strictly convex, \cite[Theorem 2.1]{LW86} implies $\mu_{n+1}(B^n) < \lambda_1(B^n)$. However, we observe numerically that the inequality $\lambda_1(B^n) < \mu_{n+2}(B^n)$ is false for $n > 3$, and prove that the number of Neumann eigenvalues below $\lambda_1(B^n)$ grows faster than any polynomial function of $n$.

The eigenvalue equation $\Delta u + \lambda u$ is separable, with radial solutions given by the so-called ultraspherical Bessel functions, $r^{1-n/2} J_{n/2+\ell-1}(\sqrt\lambda r)$, for $\ell \in \bbN \cup\{0\}$; see \cite{AB93,LS94}. The positive Neumann eigenvalues are, thus, of the form $(p_{n/2,k}^{(\ell)})^2$, where $p_{\nu,k}^{(\ell)}$ denotes the $k$th positive zero of  $[x^{1-\nu} J_{\nu+\ell-1}(x)]'$. The corresponding angular solution is a spherical harmonic of degree $\ell$, so eigenvalues corresponding to $\ell=0$ are simple, those corresponding to $\ell=1$ have multiplicity $n$, and those corresponding to $\ell > 1$ have multiplicity
\begin{align}\label{multiplicity}
	\binom{n+\ell-1}{n-1} - \binom{n+\ell-3}{n-1}.
\end{align}
(It is immediate that $(p_{n/2,k}^{(\ell)})^2$ has multiplicity at least \eqref{multiplicity}; the fact that the multiplicity is no greater is a consequence of the fact that, for any $m \in \bbN$, the functions $[x^{1-\nu} J_{\nu+\ell-1}(x)]'$ and $[x^{1-\nu} J_{\nu+\ell+m-1}(x)]'$ have no zeros in common, which was established in \cite[Lemma 2.5]{HP16}.)

Similarly, the first Dirichlet eigenvalue is given by $\lambda_1 = \big(j_{n/2-1,1}\big)^2$, where $j_{\nu,1}$ denotes the first positive zero of $J_{\nu}$. It follows from standard properties of Bessel functions (see \cite{LS94}) that $j_{\nu,k} = p_{\nu,k}^{(0)}$.


Since $j_{\nu,1}$ is an increasing function of $\nu$ (provided $\nu>0$, see \cite[p. 508]{W44}), we have $\sqrt{\lambda_1} = j_{n/2-1,1} < j_{n/2,1} = p_{n/2,1}^{(0)}$, and hence $p_{n/2,k}^{(0)} > \sqrt{\lambda_1}$ for all $k$. Moreover, it follows from Dixon's interlacing theorem \cite[p. 480]{W44} that the zeros of $[x^{1-\nu} J_{\nu+\ell-1}(x)]'$ interlace with those of $J_{\nu+\ell-1}(x)$. In particular, $J_{n/2}$ must have a zero between $p_{n/2,1}^{(1)}$ and $p_{n/2,2}^{(1)}$. Therefore $p_{n/2,2}^{(1)} > j_{n/2,1} > \sqrt{\lambda_1}$, and so $p_{n/2,k}^{(1)} > \sqrt{\lambda_1}$ for $k \geq 2$.

On the other hand, the inequality 
\begin{align}\label{LSineq}
	\frac{2\ell(\nu+\ell)(\nu+\ell+1)}{\nu+2\ell+1} < \big(p_{\nu,1}^{(\ell)}\big)^2 < 2\ell(\nu + \ell)
\end{align}
of Lorch and Szego \cite{LS94} implies that $p_{\nu,1}^{(1)} < p_{\nu,1}^{(\ell)}$ for $\nu>0$ and $\ell \geq 2$. Therefore $(p_{n/2,1}^{(1)})^2$ is the smallest positive Neumann eigenvalue of $B^n$, with multiplicity $n$, and the only other Neumann eigenvalues potentially less than $\lambda_1$ are $(p_{n/2,k}^{(\ell)})^2$ for $\ell \geq 2$. Numerically, we observe that $p_{n/2,1}^{(2)} < j_{n/2-1,1}$ for $n \geq 4$; see Table \ref{table:zeros}. Thus there are more than $n+1$ Neumann eigenvalues below the first Dirichlet eigenvalue whenever $n \geq 4$. In terms of the quantity $N$ defined in \eqref{eq:Ndef}, this means $N(B^n) > n+1$ when $n \geq 4$.

\begin{table}
\begin{center}
\begin{tabular}{c|c|c|c|c}
	$n$ & $p_{n/2,1}^{(1)}$ & $p_{n/2,1}^{(2)}$ & $p_{n/2,1}^{(3)}$ & $j_{n/2-1,1}$\\   \hline 
	2 & 1.84 & 3.05 & 4.42 & 2.40 \\  
	3 & 2.08 & 3.34 & 4.51 & 3.14 \\
	4 & 2.30 & 3.61 & 4.81 & 3.83 \\
	5 & 2.52 & 3.86 & 5.09 & 4.49 \\
	6 & 2.69 & 4.10 & 5.37 & 5.14 \\
	7 & 2.86 & 4.33 & 5.63 & 5.76 \\ \hline
\end{tabular}
\end{center}
\caption{Comparing zeros of ultraspherical Bessel functions with the first zero of $J_{n/2-1}$}
\label{table:zeros}
\end{table}

More precisely, when $p_{n/2,1}^{(2)} < j_{n/2-1,1}$ we obtain an additional Neumann eigenvalue below $\lambda_1$. The angular part of the eigenfunction is a spherical harmonic of degree $\ell$, so the multiplicity is
\[
	\binom{n+1}{n-1} - \binom{n-1}{n-1} = \frac12 n(n+1) - 1.
\]
Adding this to the $n+1$ Neumann eigenvalues already known to exist below $\lambda_1$, we find that $N(B^n) \geq \frac12 n(n+3)$.

This argument can be generalized, using the fact that for any fixed $\ell$, we have $p_{n/2,1}^{(\ell)} < j_{n/2-1,1}$ once $n$ is sufficiently large. (For $\ell=3$ we observe numerically that $n=7$ suffices; see Table \ref{table:zeros}.) As a result, we find that $N(B^n)$ grows faster than any polynomial function in $n$.

\begin{theorem}\label{thm:ball}
For any $\ell \in \bbN$ there exists $n_0 \in \bbN$ such that $N(B^n) \geq n^\ell$ for all $n \geq n_0$.
\end{theorem}

\begin{proof}
From the upper bound in \eqref{LSineq} and the lower bound $j_{\nu,1} > \nu$ (see \cite[p. 485]{W44}), we find that $\big(p_{n/2,1}^{(2)}\big)^2 < \lambda_1$ if $2\ell(n/2 + \ell) \leq (n/2-1)^2$. For fixed $\ell$, this will be satisfied by all sufficiently large $n$. For any such $n$, there is thus a Neumann eigenvalue below $\lambda_1$, of multiplicity
\[
	\binom{n+\ell-1}{n-1} - \binom{n+\ell-3}{n-1} = \frac{n^\ell}{\ell!} + \text{ lower order terms},
\]
so for large enough $n$, we can guarantee $N(B^n) \geq n^\ell/(2 \ell!)$.

Applying the above argument to $\ell+1$, we obtain $N(B^n) \geq n^{\ell+1}/(2 (\ell+1)!)$ when $n$ is sufficiently large, say $n \geq n_1$. Thus, for $n \geq \max\{n_1,2(\ell+1)!\}$ we have $N(B^n) \geq n^\ell$.
\end{proof}

%
%
%

\section{Numerical results}\label{sec:numerical}
In this section, we present the numerical results that motivated Conjecture \ref{con:2D}.

Using a finite-element method (FEM), we compute $N(\Omega)$ and plot it against $\cI(\Omega)$ for different choices of $\Omega$. Note that the size of $\Omega$ is irrelevant as both quantities are scale invariant.

\subsection{Overview of FEM} 
In general, a finite-element method approximates solutions to either a PDE or the associated eigenvalue problem by projecting the weak form of the problem onto a finite-dimensional subspace \cite{DBraess_2001, SCBrenner_LRScott_1994a}.  For a symmetric and positive-definite operator, as considered here, a standard Galerkin projection is typically chosen, with the finite-dimensional approximation space determined by low-order piecewise polynomial basis functions on an appropriate discretization of the domain, $\Omega$.  Considering the weak form of the Laplacian eigenvalue problem, finding $u\in \mathcal{V} \subset H^1(\Omega)$ such that
\[
\int_\Omega \nabla u \cdot \nabla v = \lambda \int_\Omega uv \text{ for all }v\in\mathcal{V},
\]
we define a space $\mathcal{V}^h \subset \mathcal{V}$ in terms of a finite-dimensional basis, and simply solve the eigenvalue problem restricted to $\mathcal{V}^h$ in place of the continuum weak form.  Once restricted to a finite-dimensional basis set, the approximate eigenvalue calculation becomes a standard matrix eigenvalue problem, for which efficient numerical methods are well-known \cite{YSaad_2011a}. The form above is equally valid for both the Dirichlet and Neumann eigenvalue problems (or combinations of these boundary conditions), simply by appropriately choosing $\mathcal{V} = H^1_0(\Omega)$ and $\mathcal{V} = H^1(\Omega)$, respectively. For the Dirichlet problem, the boundary condition is enforced strongly, as it is automatically satisfied by all test functions in $\mathcal{V}$, whereas, for the Neumann problem, the boundary condition is enforced weakly.


In the calculations that follow, we consider only polygonal domains, $\Omega$.  As such, we directly discretize the domain by considering triangulations, $\mathcal{T}^h$, such that $\Omega = \cup_{T\in\mathcal{T}^h} T$, where the union is taken over planar triangles.  The superscript $h$ is given to indicate the discrete nature of the triangulation (and the associated approximation space, $\mathcal{V}^h$), and may be taken to be the maximum edge length (or diameter) of any triangle in $\mathcal{T}^h$, for example.  The approximation space, $\mathcal{V}^h$, is then defined as
\[
\mathcal{V}^h = \{ u \in C^0(\Omega) \mid u_T \in \mathcal{P}^k(T) \text{ for every }T\in\mathcal{T}^h \},
\]
where $u_T$ is the restriction of $u$ to triangle $T$, and $\mathcal{P}^k(T)$ is the space of polynomials of total degree no more than $k$ on $T$.  The calculations below are computed using the FEniCS finite-element package \cite{MAlnaes_etal_2015a}, with eigenvalue calculations performed using the SLEPc package \cite{Hernandez:2005:SSF}.

\subsection{Rectangles}
Let $R_\ell$ denote the rectangle with side lengths of $1$ and $\ell \geq 1$. Our numerical results, shown in Figure  \ref{fig:rectanglecomb}, are consistent with the explicit formulas for $N(R_\ell)$ and $\cI(R_\ell)$ given in \eqref{rec:2D}. In particular, the graph of $N$ vs. $\cI$ is asymptotic to a line of slope $1/4$ as $\cI \to \infty$.

\begin{figure}
	\includegraphics[scale=0.45]{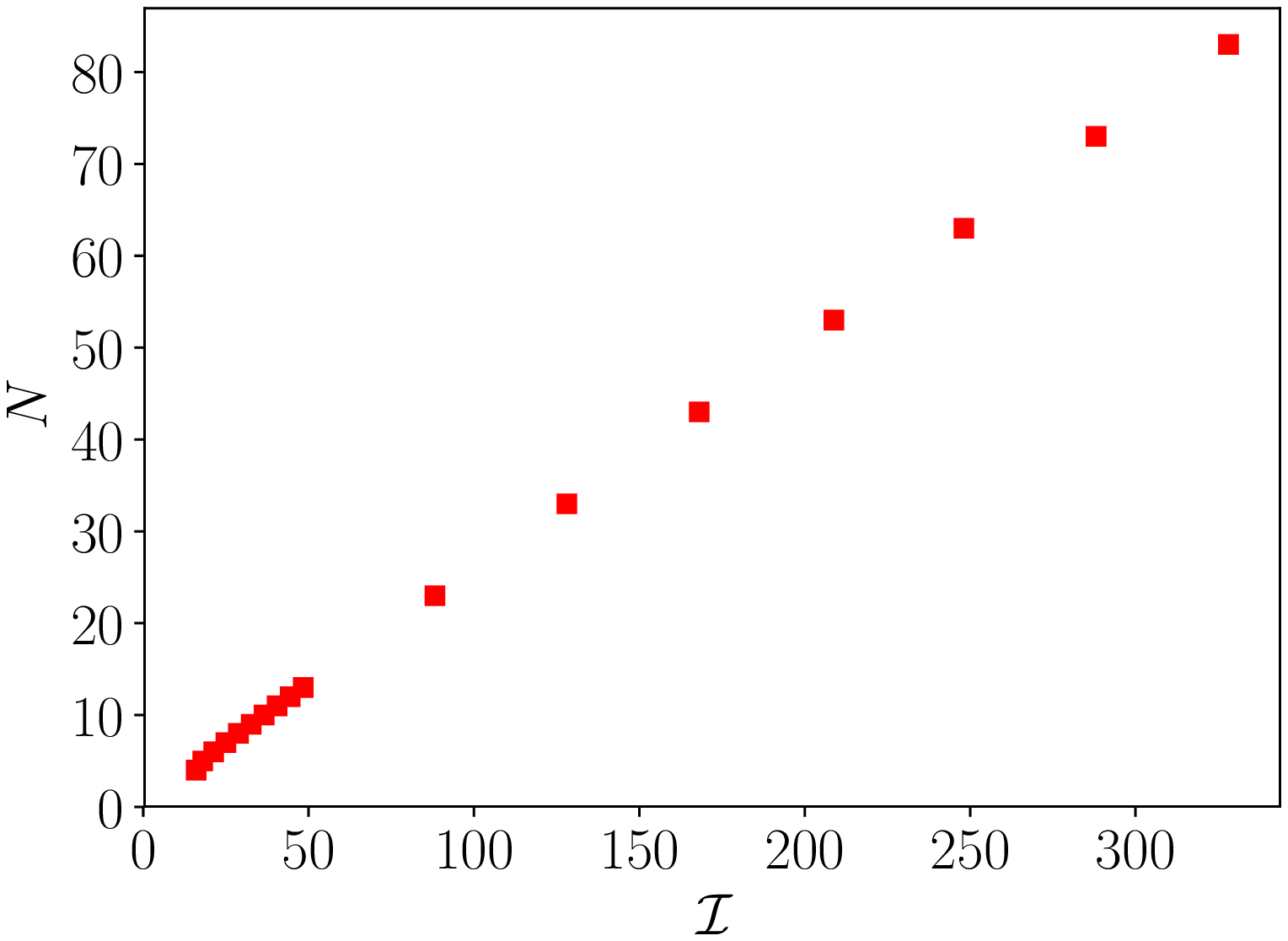}
	\includegraphics[scale=0.45]{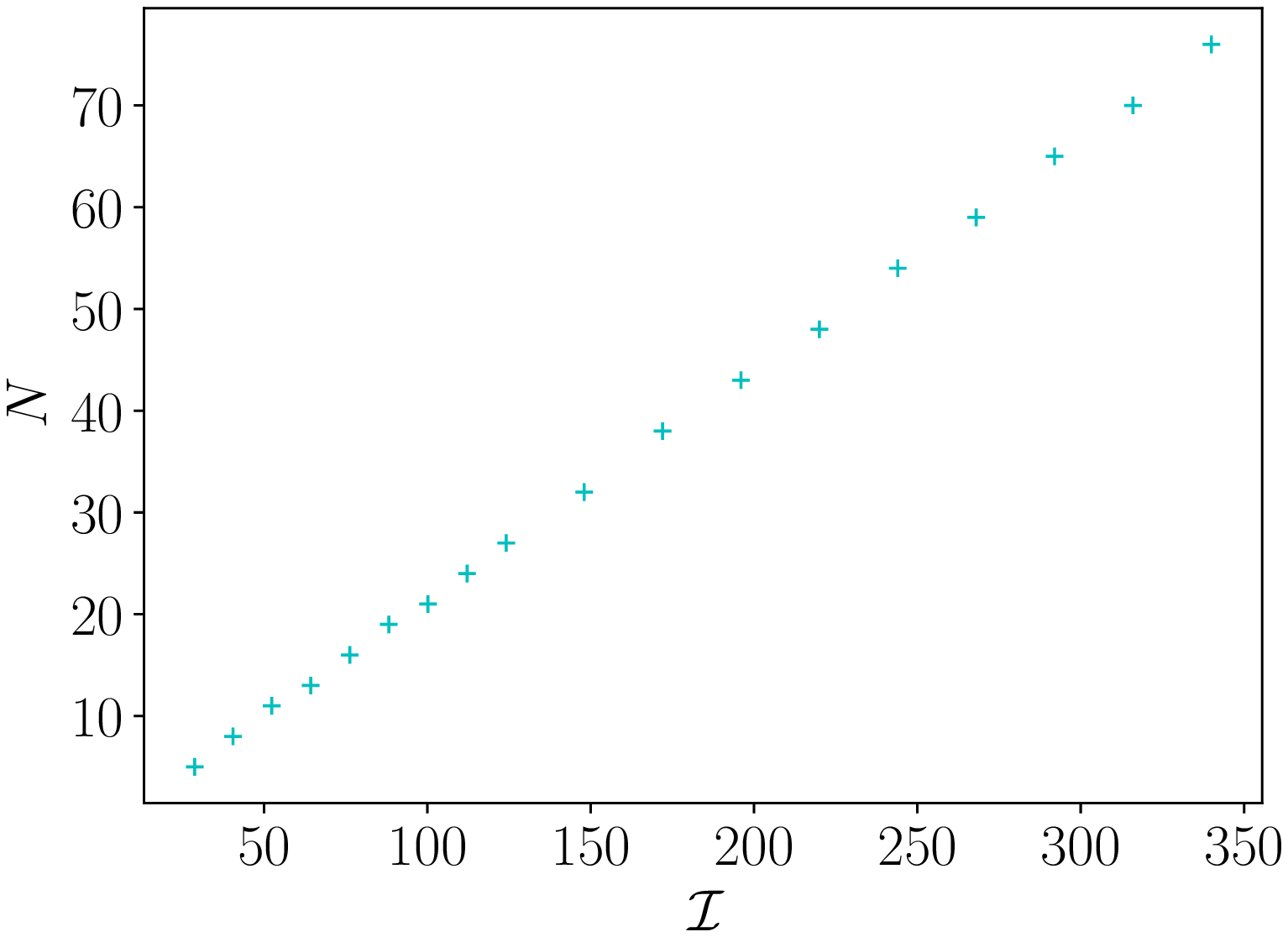}
\caption{$N$ vs. $\cI$ for rectangles (left) and combs (right)}
\label{fig:rectanglecomb}
\end{figure}

\subsection{Combs}
We next consider a family of so-called ``comb" domains. The comb with $m$ teeth, denoted $C_m$, is the union of $m$ $1\times 2$ rectangles with $m-1$ squares with unit side length, as shown in Figure \ref{fig:CW}. It has area $3m-1$ and perimeter $6m$; hence, the isoperimetric ratio $\cI(C_m) = 36 m^2/(3m-1)$ is approximately linear in $m$. The corresponding values of $N(C_m)$ are shown in Figure \ref{fig:rectanglecomb}.

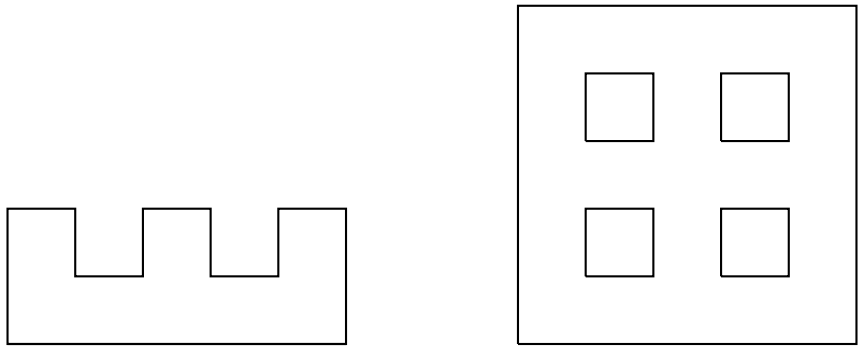
\begin{figure}
\begin{tikzpicture}[scale=0.9]
	\draw[thick] (0,0) -- (5,0) -- (5,2) -- (4,2) -- (4,1) -- (3,1) -- 
		(3,2) -- (2,2) -- (2,1) -- (1,1) -- (1,2) -- (0,2) -- (0,0) -- (5,0);
\end{tikzpicture}
\hspace{2cm}
\begin{tikzpicture}[scale=0.9]
	\draw[thick] (0,0) -- (5,0) -- (5,5) -- (0,5) -- (0,0);
	\draw[thick] (1,1) -- (1,2) -- (2,2) -- (2,1) -- (1,1);
	\draw[thick] (3,1) -- (3,2) -- (4,2) -- (4,1) -- (3,1);
	\draw[thick] (1,3) -- (1,4) -- (2,4) -- (2,3) -- (1,3);
	\draw[thick] (3,3) -- (3,4) -- (4,4) -- (4,3) -- (3,3);
\end{tikzpicture}
\caption{The comb $C_3$ (left) and the waffle $W_2$ (right)}
\label{fig:CW}
\end{figure}

\begin{figure}
	\includegraphics[scale=0.5]{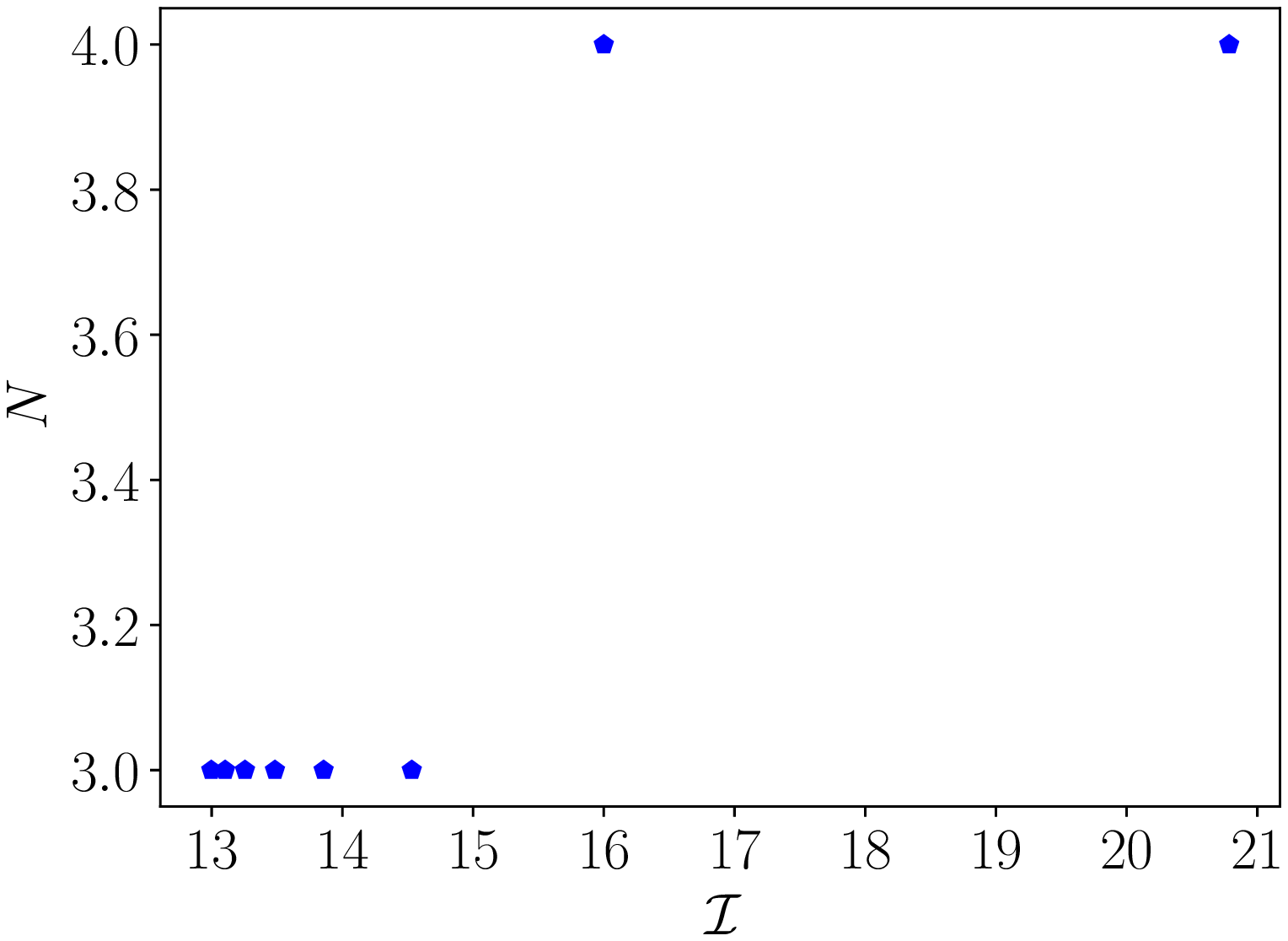}
	\includegraphics[scale=0.5]{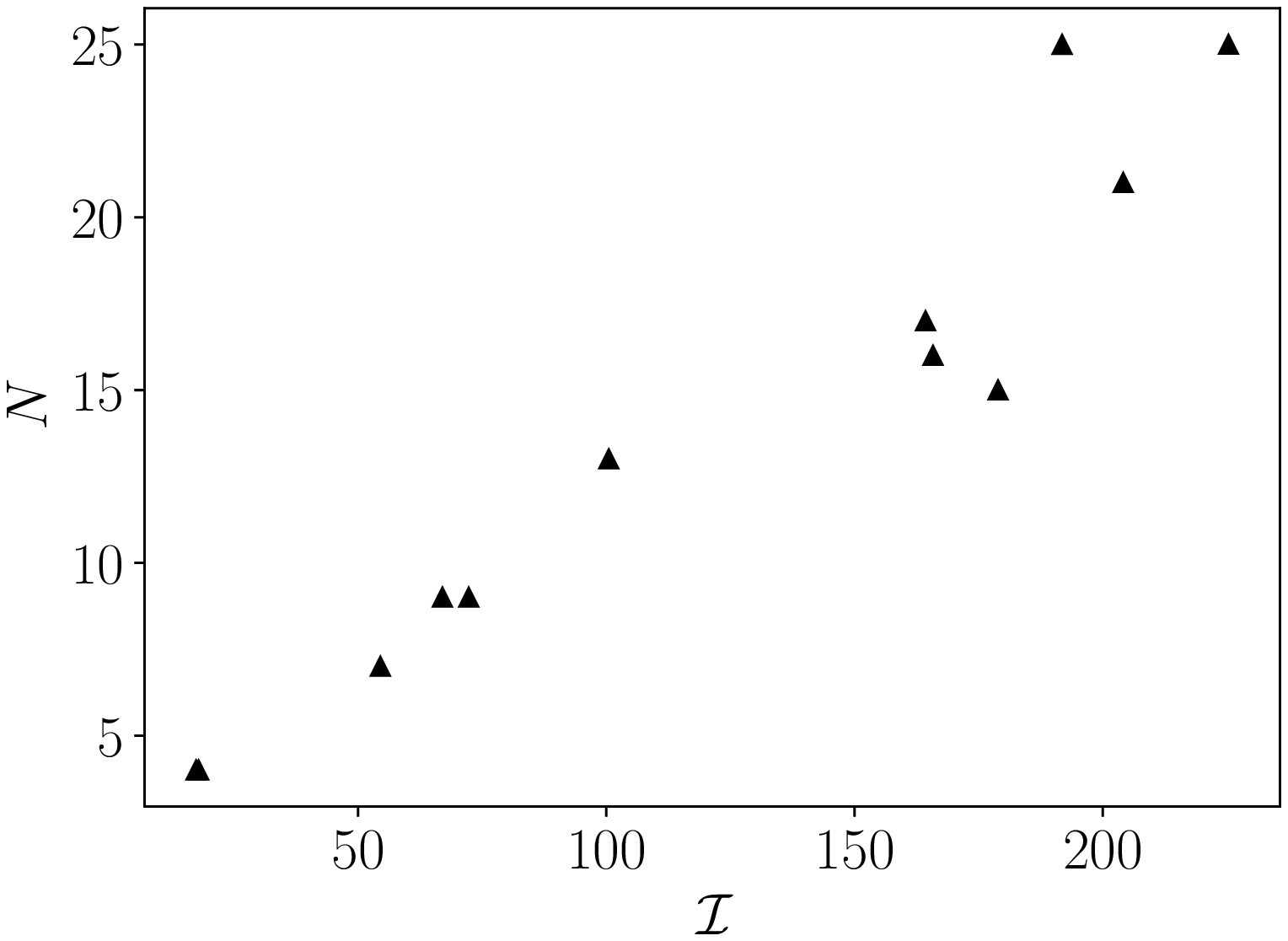}
\caption{$N$ vs. $\cI$ for regular polygons (left) and random polygons (right)}
\label{fig:poly}
\end{figure}

\subsection{Regular polygons}
Let $P_m$ be a regular $m$-sided polygon. One easily calculates the isoperimetric ratio, $\cI(P_m) = 4m \tan(\pi/m)$, which decreases towards $4\pi$ as $m\to\infty$. On the other hand, it is also known (see, for instance, \cite[Theorem VI.10]{CH53}) that the Dirichlet and Neumann eigenvalues satisfy
\[
	\lambda_k(P_m) \to \lambda_k(D), \quad \mu_k(P_m) \to \mu_k(D)
\]
as $m \to \infty$, where $D \subset \bbR^2$ is the unit disc. Since $\mu_3(D) < \lambda_1(D) < \mu_4(D)$, we obtain $\mu_3(P_m) < \lambda_1(P_m) < \mu_4(P_m)$, and hence $N(P_m) = 3$, for all sufficiently large $m$. We observe numerically that it suffices to take $m = 5$; see Figure \ref{fig:poly}.

\subsection{Random polygons}\label{sec:random}
We next consider a family of random polygons, generated by the following algorithm.

First, three distinct random points within the unit square are chosen and ordered counter clockwise in a vertex list: $(v_1,v_2,v_3)$. Then, until the desired number of vertices is achieved, new vertices are added as follows: 
\renewcommand{\labelenumi}{\arabic{enumi})}
\begin{enumerate}
	\item Generate a random point $P$ distinct from the existing vertices $(v_1,\ldots,v_m)$, and select a random index $i \in \{1,\ldots,m\}$ for a position in the current list of vertices.
	\item Check if creating edges between $P$ and the vertices $v_i$ and $v_{i+1}$ would cause any edges in the new polygon to intersect. If no edges intersect, then $P$ is added to the vertex list between $v_i$ and $v_{i+1}$.
	\item If an intersection occurs, replace $i$ with $i+1 \ (\operatorname{mod} m)$ and go to step 2.
	\item If $P$ cannot be added between any adjacent vertices without causing an intersection, return to step 1. (The necessity of this step is demonstrated by Figure \ref{fig:counter}.)
\end{enumerate}

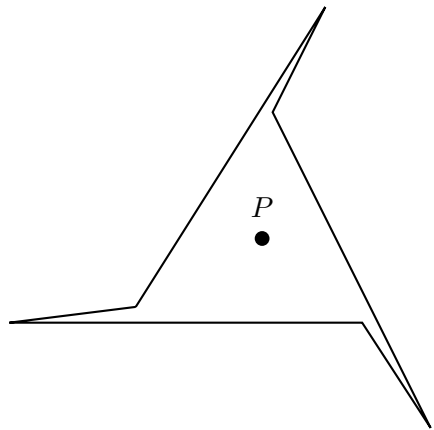
\begin{figure}
\begin{tikzpicture}[scale=1.4]
	\draw[thick] (-0.8,-0.85) -- (1,2) -- (0.5,1) -- (2,-2) -- (1.35,-1) -- (-2,-1) 
	-- (-0.8,-0.85);
	\fill (0.4,-0.2) circle[radius=2pt];
	\node at (0.4,0.1) {$P$};
\end{tikzpicture}
\caption{The point $P$ cannot be added between any adjacent vertices without causing an intersection in the resulting seven-sided polygon}
\label{fig:counter}
\end{figure}

Figure \ref{fig:randomsteps} tracks the evolution of $N$ and $\cI$ throughout this process, for two different realizations of this algorithm. The resulting domains tend to be highly non-convex; some representative 30-sided examples are shown in Figure \ref{fig:random}. In Figure \ref{fig:poly}, we plot $N$ vs. $\cI$ for twelve such random polygons, generating two each with 5, 10, 15, 20, 25 and 30 sides.


\begin{figure}
	\includegraphics[scale=0.5]{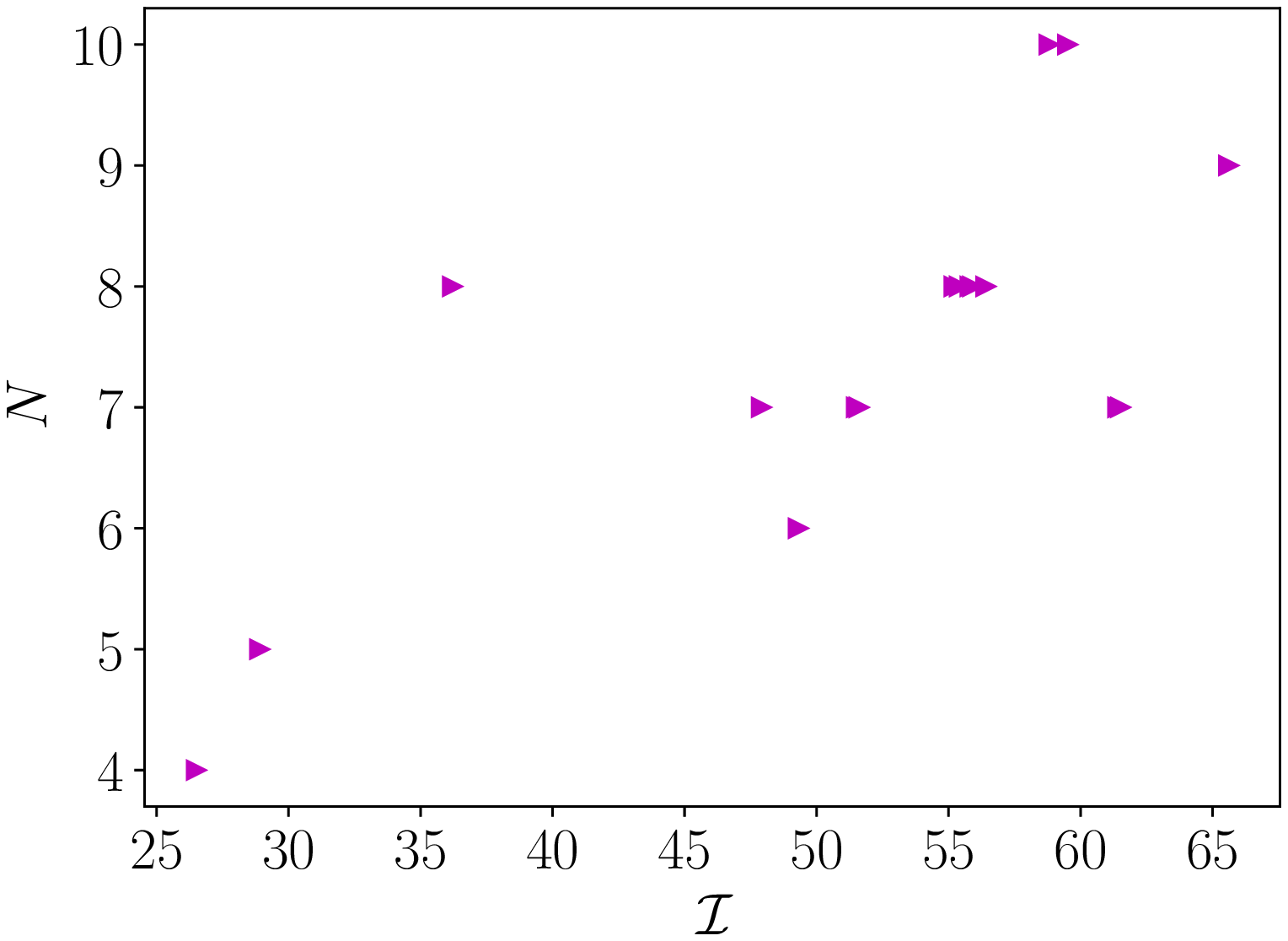}
		\includegraphics[scale=0.5]{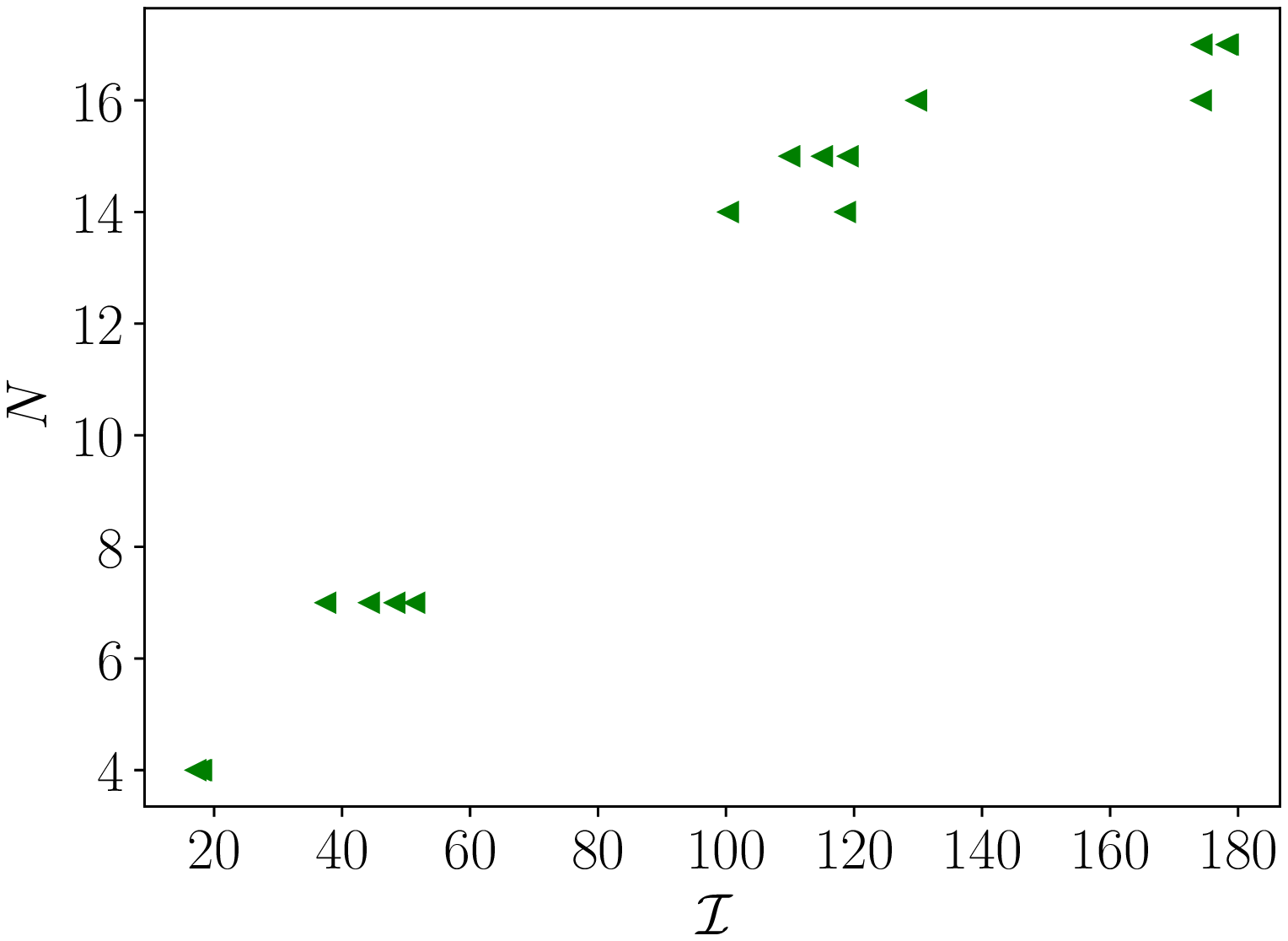}
\caption{Evolution of $N$ and $\cI$ (as vertices are added) for two realizations of the random polynomial generator}
\label{fig:randomsteps}
\end{figure}

\begin{figure}
	\includegraphics[scale=0.5]{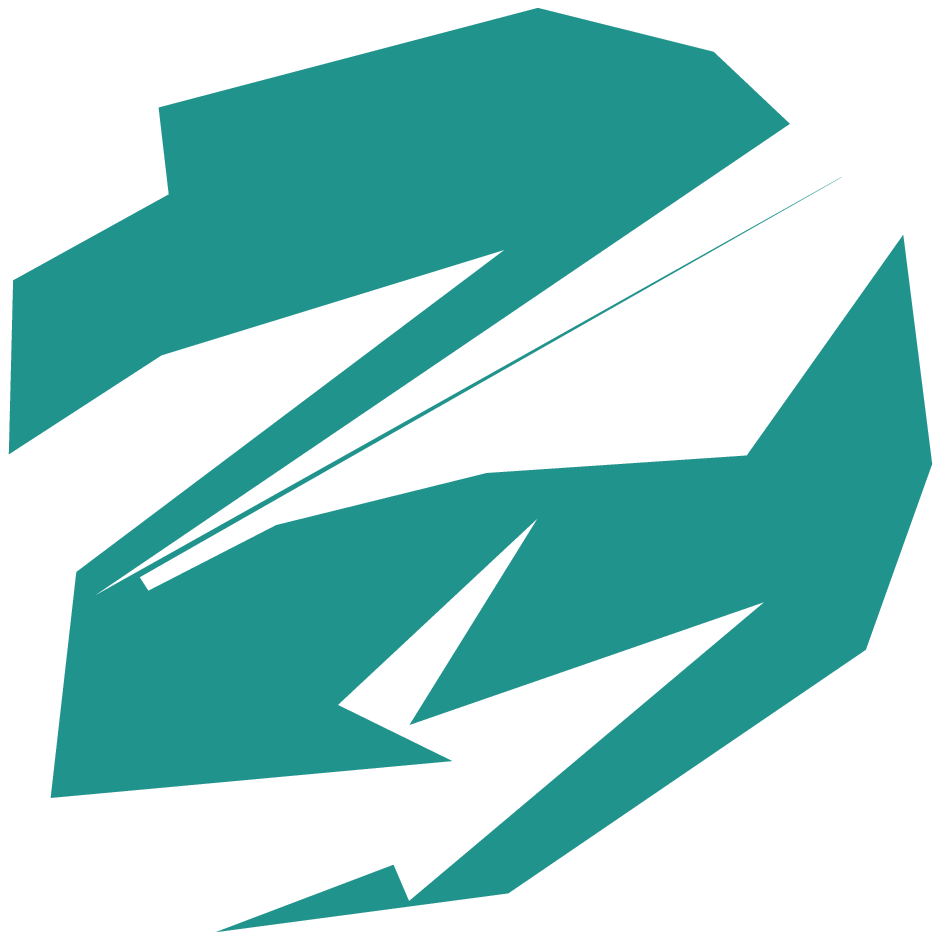}
		\includegraphics[scale=0.5]{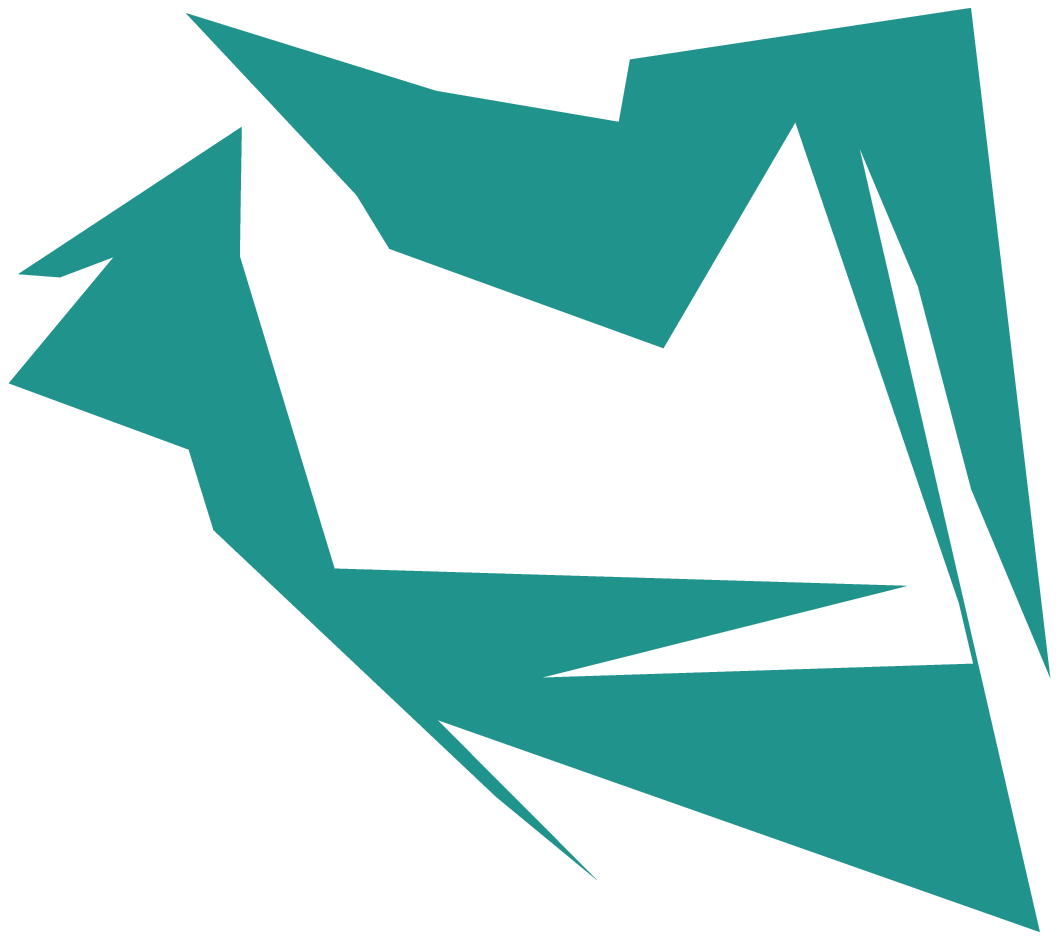}
\caption{Two of the randomly generated polygons considered in Section \ref{sec:random}}
\label{fig:random}
\end{figure}

\subsection{Non-simply connected domains}
The domains considered above were all simply connected. In this section, we consider two families of domains with holes. In both cases, we observe behaviour consistent with Conjecture \ref{con:2D}, providing evidence that the conjectured bounds on $N(\Omega)$ hold uniformly for $\Omega \subset \bbR^2$, independent of its topology.

The first is the so-called ``square annulus," consisting of a unit square with a smaller concentric square removed from its interior. The results are shown in Figure \ref{fig:annulus}, where the interior side length ranges from 0.1 to 0.9.

\begin{figure}
	\includegraphics[scale=0.5]{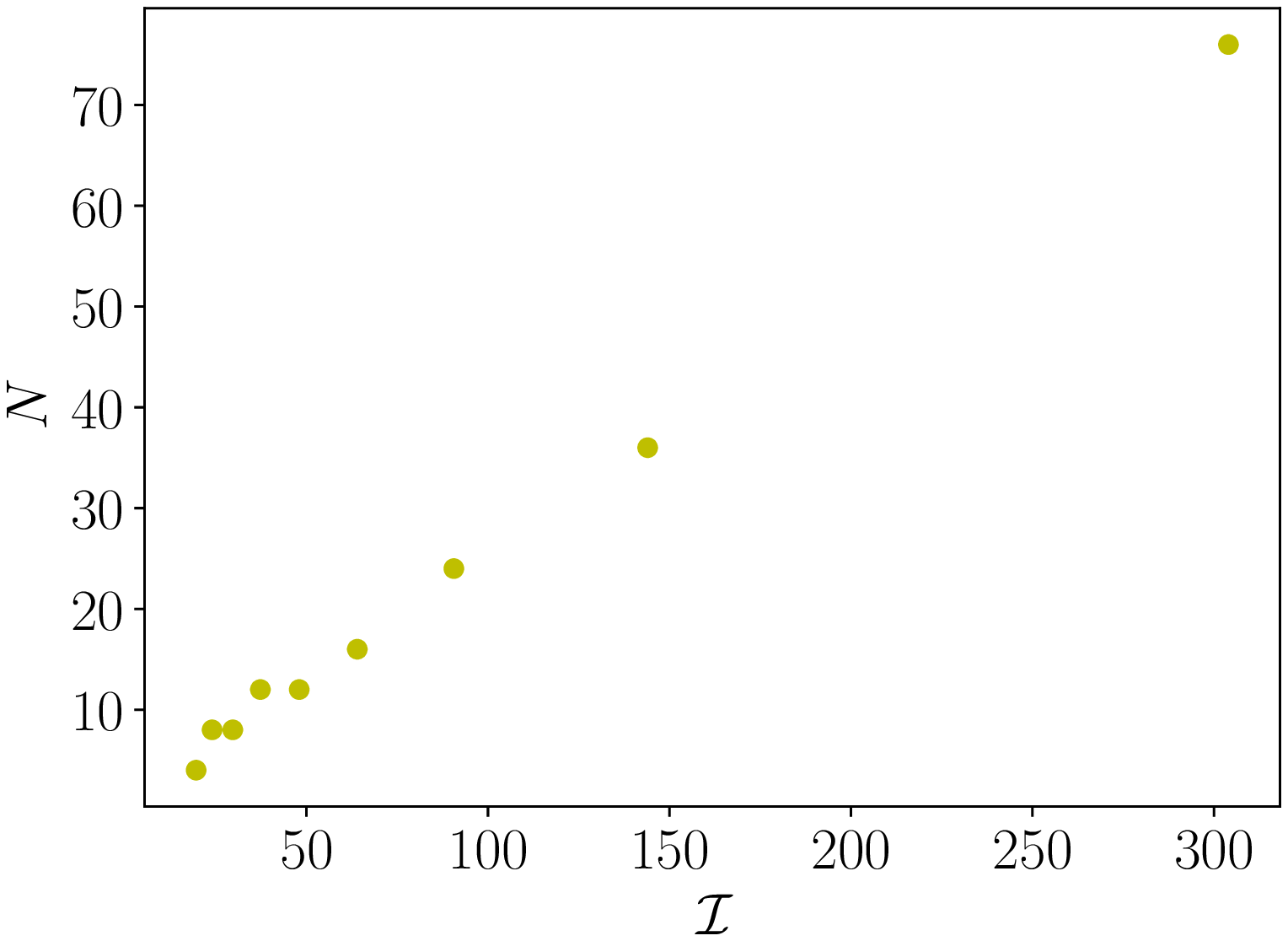}
	\includegraphics[scale=0.5]{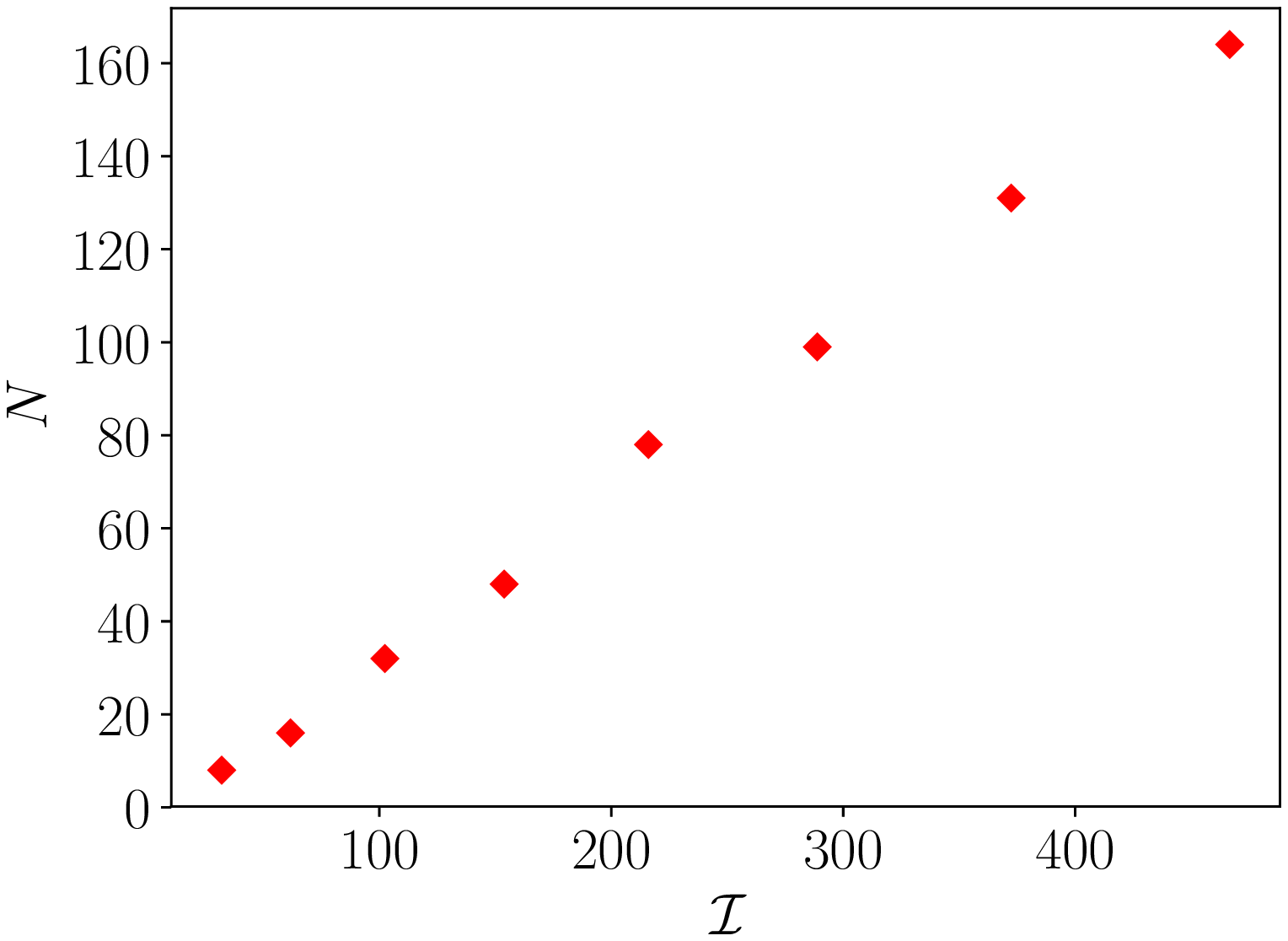}
\caption{$N$ vs. $\cI$ for square annuli (left) and waffles (right)}
\label{fig:annulus}
\end{figure}

The next example is the family of ``waffle" domains. The $m$th waffle domain, $W_m$, consists of a square of side length $2m+1$ with $m^2$ evenly spaced unit squares removed from its interior; see Figure \ref{fig:CW} for an illustration of a waffle and Figure \ref{fig:annulus} for the numerical results with $1 \leq m \leq 8$.

\section{Conclusion}
In this paper, we numerically approximated the quantity $N(\Omega)$ for many planar domains, with varying geometry and topology. In all cases, we observed that $N(\Omega)$ is controlled by the isoperimetric ratio, $\cI(\Omega)$. Based on these observations, we hypothesized that this relationship always holds (Conjecture~\ref{con:2D}). We also suggested that this holds in higher dimensions (Conjecture \ref{con:general}). We also discussed some implications these conjectures would have if they are indeed true. In particular, our conjectures, combined with results from \cite{CJM17}, yield a direction connection between the spectral position of an eigenfunction and the isoperimetric ratio of its nodal sets. 

We suggest that these conjecture be approached using the Dirichlet-to-Neumann map formalism in \cite{CJM17,F91}. This approach seems promising in light of results giving isoperimetric control on Steklov eigenvalues; see, for instance, \cite{CEG11}.

In terms of the finite-element experiments presented in this paper, a natural generalization would be the numerical study of domains in higher dimensions, where the geometry and topology can be much more complicated than in the planar case studied here.

\bibliographystyle{amsplain}
\bibliography{DNeigenvalue}

\end{document}